\documentclass[11pt,a4paper,leqno,twoside]{amsart}
\usepackage{amsmath,amssymb,amsthm,enumerate,hyperref,color}
\numberwithin{equation}{section}
\newtheorem{theorem}{Theorem}[section]
\newtheorem{definition}[theorem]{Definition}

\newtheorem{proposition}[theorem]{Proposition}
\newtheorem{lemma}[theorem]{Lemma}
\newtheorem{corollary}[theorem]{Corollary}
\newtheorem{remark}[theorem]{Remark}

\newcommand{\nor}{\Arrowvert}

\def\e{{\varepsilon}}

\def\na{\nabla}

\def\g{{\gamma}}

\def\L{{\Lambda}}
\def\l{{\lambda}}
\def\a{{\alpha}}

\newcommand{\R}{\mathbb{R}}
\newcommand{\N}{\mathbb{N}}

\newcommand{\jcal}{{\mathcal{I}}}
\newcommand{\kcal}{{\mathcal{K}}}
\newcommand{\n}{\mathop{\textsc{n}}}

\newcommand{\const}{\mathop{\textsc{c}}}
\newcommand{\bsc}{\mathop{\textsc{b}}}
\newif\ifcomment \commentfalse
\def\commentON{\commenttrue}

\long\outer\def\BC#1\EC{\ifcomment \sloppy \par \# \ldots\dotfill
{\em #1} \dotfill \# \par \fi } \commentON

\newcommand{\remove}[1]{}

\begin{document}
\title[Nonradial solutions]{A nonradial bifurcation result with applications to supercritical problems}
\author[A.~L.~Amadori, F.~Gladiali]{Anna Lisa Amadori$^\dag$ and  Francesca Gladiali$^\ddag$}
\thanks{The authors are members of the Gruppo Nazionale per
l'Analisi Matematica, la Probabilit\'a e le loro Applicazioni (GNAMPA)
of the Istituto Nazionale di Alta Matematica (INdAM). The second author is supported by PRIN-2012-grant ``Variational and perturbative aspects of nonlinear differential problems''.}
\date{\today}
\address{$\dag$ Dipartimento di Scienze Applicate, Universit\`a di Napoli ``Parthenope", Centro Direzionale di Napoli, Isola C4, 80143 Napoli, Italy. \texttt{annalisa.amadori@uniparthenope.it}}
\address{$\ddag$ Matematica e Fisica, Polcoming, Universit\`a di Sassari, via Piandanna 4, 07100 Sassari, Italy. \texttt{fgladiali@uniss.it}}

\begin{abstract}
\noindent
In this paper we consider the problem
\[
\left\{\begin{array}{ll}
-\Delta u=|x|^{\a} F(u) & \hbox{ in }\R^{\n}\\
u>0& \hbox{ in }\R^{\n} 
\end{array}\right.
\]
where $\a>0$ and $\n\ge3$. Under some assumptions on $F$ we deduce the existence of nonradial solutions which bifurcate from the radial one when $\alpha$ is an even integer.

{\bf Keywords:} semilinear elliptic equations, bifurcation, nonradial solutions.

{\bf AMS Subject Classifications:} 35B32, 35J61.
\end{abstract}

\maketitle
\section{Introduction}\label{s1}
In this paper we consider the problem
\begin{equation}\label{1}
\left\{
\begin{array}{ll}
-\Delta u=|x|^{\a} F(u) & \hbox{ in }\R^{\n} ,\\
u>0& \hbox{ in }\R^{\n} ,
\end{array}\right.
\end{equation}
where $\alpha>0$ is a real parameter and $\n\ge 3$. Throughout the paper 
 $F$ is a function in  $ C^{1}[0,+\infty)$. We will consider the boundary conditions
\begin{equation}\label{bc1}
u(x)\to 0 \quad \hbox{ as }|x|\to +\infty.
\end{equation}
This condition is motivated for physical reasons, indeed one wants the Lagrangian associated to problem \eqref{1} to be finite. Instead of \eqref{bc1} one can impose that $u$ belongs to a suitable space, for example $H^1(\R^{\n})$ or $D^{1,2}(\R^{\n})$, depending on $F$. In most cases this implies that solutions  have an exponential decay at infinity,
and guarantees that \eqref{1} has only radial solutions for $\a=0$, see \cite{gnn} or \cite{DG}.
When $\a>0$ instead, the presence of the term $|x|^\a$ allows the existence of nonradial solutions. This phenomenon has been shown (in the case of a spherical bounded domain) in the seminal paper of \cite{SSW} for the H\'enon problem, $F(u)=u^p$ with $1<p<\frac{\n+2+2\alpha}{\n-2}$. See  \cite{AG} for some nonradial bifurcation results for the H\'enon problem in the unit ball, or \cite{BK14} for some more general nonhomogeneous nonlinearities. 
One of the few result in all of $\R^{\n}$ is \cite{GGN}, where the authors consider the critical H\'enon problem, i.e $F(u)=u^{\frac{\n+2+2\a}{\n-2}}$, and prove the existence of infinitely many nonradial bifurcation points.
Here we consider the more general problem \eqref{1} and  we show  the existence of nonradial solutions to \eqref{1} \eqref{bc1}, by using the bifurcation theory, which is a good tool that can be applied also in the supercritical case, i.e.~when $F(u)$ can growth faster that $u^{\frac{\n+2}{\n-2}}$ at infinity.

For what concerns radial solutions, problem  \eqref{1}, \eqref{bc1} is equivalent to the ordinary differential equation
\begin{equation}\label{1radtr}
- V^{\prime\prime}- \dfrac{k(\alpha)}{t} V^{\prime}= F(V) , \qquad  t>0 , 
\end{equation}
where 
\[ k(\alpha) =\frac{2\n-2+\a}{2+\a}\in(1,\n-1),\]
 with the boundary conditions
\begin{equation}\label{bc}
V'(0)=0\,,\qquad
\lim\limits_{t\to+\infty}V(t) =0.
\end{equation}
It is very easy, indeed, to prove next statement.
\begin{proposition}\label{p:radial} 
Assume that problem \eqref{1radtr}, \eqref{bc} admits a solution $V_{\alpha}$ for some $\a>0$. Then problem \eqref{1}, \eqref{bc1} has the radial solution
\begin{equation}\label{ualfa}
u_{\alpha}(x)= V_{\alpha}\left(\frac{2}{2+\alpha} |x|^{\frac{2+\alpha}{2}}\right).
\end{equation}
Moreover any radial solution of \eqref{1}, \eqref{bc1} gives rise to a solution to \eqref{1radtr}, \eqref{bc} by the formula \eqref{ualfa}.
\end{proposition} 
This relation comes from the transformation 
\begin{equation}\label{cov}
 t=\dfrac{2}{2+\alpha} r^{(2+\alpha)/2},
\end{equation}
which maps any radial solution $u_{\alpha}$ of \eqref{1}, \eqref{bc1} to a solution $V_{\alpha}$ of \eqref{1radtr}, \eqref{bc}. 
The equivalence between the radial solutions of \eqref{1} and the solutions to \eqref{1radtr} applies to any type of solutions, positive or sign-changing, and holds for bounded domains also.
It has been already used in \cite{GGN} and \cite{GGN2}, and shall be exploited also here. 
When $k(\a)$ is an integer problem \eqref{1radtr}, \eqref{bc} is equivalent to find radial solutions for the autonomous problem
\begin{equation}\label{BLauto}\left\{\begin{array}{ll}
-\Delta v =F(v) &\hbox{ in }\R^k,\\
v\to 0 &\hbox{ as } |x|\to \infty,
\end{array}\right.\end{equation}
where the new dimension $k$ is strictly less than $\n$. 
In this way, the transformation \eqref{cov} relates solutions to \eqref{1} to solutions to an homogeneous problem in a space of lower dimension, and so allows to obtain radial solutions when $F$ grows faster than $\frac{\n+2}{\n-2}$ at infinity but slower than $\frac{k+2}{k-2}$.

In addition,  here, we draw information about the bifurcation points in the curve $\a\mapsto u_{\a}$ of classical radial solutions to \eqref{1}, \eqref{bc1}.
By the implicit function theorem, the bifurcation values of $\a$ must satisfy the degeneracy condition
\begin{equation}\label{L}
-\Delta w=|x|^{\a} F'(u_{\a})\,w \qquad  \hbox{ in }\R^{\n},
\end{equation}
for some $w$ such that $|x|^{\alpha/2}w(x) ,|\na w(x)| \in L^2(\R^{\n})$.
In general solving \eqref{L}, to detect if $u_{\a}$ is degenerate or not, is a very difficult problem. However we are able to solve it, and explicitly find the values ​​of $\alpha$ where the degeneration appears and the corresponding eigenfunctions, when  the solution $V_{\a}$ of \eqref{1radtr} has Morse index one.
To give  precise statements, we introduce the functional space related to the linearization of \eqref{1radtr}:
\begin{equation}\label{E} 
E=\left\{ V\in C^1(0,+\infty)\,:\, \int_0^{+\infty}t^k\left( (V')^2+V^2\right)\, dt<\infty\right\}. 
\end{equation}
Let us recall that the solution $V_\a$ of \eqref{1radtr}, \eqref{bc} has Morse index one if 
$$\inf\limits_{\begin{subarray}{c}{w\in E,}\\{w\neq 0}\end{subarray}}\frac{\int_0^{+\infty}t^k\left((w')^2-F'(V_{\a})w^2\right) dt}{ \int_0^{+\infty}t^kw^2\, dt}<0$$
and
$$\inf\limits_{\begin{subarray}{c}{W\subset E,}\\{dim(W)=2}\end{subarray}}\max_{\begin{subarray}{c}w\in W\\ w\neq 0\end{subarray}}\frac{\int_0^{+\infty}t^k\left((w')^2-F'(V_{\a})w^2\right) dt}{ \int_0^{+\infty}t^kw^2\, dt}\geq 0.$$
Moreover $V_\a$ is said nondegenerate if the linearized equation
\begin{equation}\label{kkk}
-w''-\frac{k}tw'=F'(V_{\a})w\qquad t>0
\end{equation}
does not have nontrivial solutions in $E$. If $V_\a$ is degenerate we let $n_\a$ be the dimension of the space of solutions to \eqref{kkk} in $E$. Obviously $n_\a=1$ or $2$. 

\begin{proposition}\label{lin}
Assume that, for some  $\a>0$, \eqref{1radtr} has a positive solution $V_{\a}$ with Morse index 1  and  $V_{\a}'\in E$. 
\begin{enumerate}[(i)]
\item 
If $V_{\a}$ is nondegenerate, then the linearized problem \eqref{L} 
has a nontrivial solution if and only if $\alpha=2 i$ is an even integer. Moreover the space of solutions of \eqref{L} has dimension $\frac{(\n+2i)(\n+i-2)!}{(\n-2)!(i+1)!}$ and it is spanned by
\[
w_{i,h}(x) = V_{\alpha}'\left(\frac{2}{2+\alpha} |x|^{\frac{2+\alpha}{2}}\right)\, Y_{i+1,h}(\theta), \]
where $Y_{j,h}$ are the spherical harmonic functions related to the $j^{th}$ eigenvalue of the Laplace-Beltrami operator in $S^{\n-1}$ ($S^{\n-1} $ denotes the $\n-1$-dimensional sphere.)
\item 
If $V_{\a}$ is  degenerate and $\a$ is not an even integer, then the space of solutions to 
the linearized problem \eqref{L} has dimension $n_\a$ and  it is spanned by 
\[w_0(x) = \psi\left(\frac{2}{2+\alpha} |x|^{\frac{2+\alpha}{2}}\right) , \]
where $\psi\in E$ is any solution of the problem \eqref{kkk}.
\item 
Otherwise, if $V_{\a}$ is  degenerate and $\a=2i$ is an even integer, then the space of solutions has dimension $n_\a+\frac{(\n+2i)(\n+i-2)!}{(\n-2)!(i+1)!}$  and it is spanned by $w_0$ and $w_{i,h}$.
\end{enumerate}
\end{proposition}

An interesting consequence of Proposition \ref{lin} is the computation of the Morse index of the radial solution $u_{\a}$. 
\begin{corollary}\label{cor1}
Assume that \eqref{1radtr} has a positive solution $V_{\a}$ with Morse index 1. Then the Morse index of $u_{\alpha}$ 
is given by 
\[ m(\alpha) = \sum_{0\leq i<1+\alpha/2}\frac{(\n+2i-2)(\n+i-3)!}{(\n-2)!i!}
,\]
where $i$ runs through the set of integer numbers.
\end{corollary}

The result of Proposition \ref{lin} can be used to prove some existence results for solutions of problem \eqref{1} in a bounded domain, see \cite{GG} as an example, or in the study of some perturbed equations.\\
Now, provided that problem \eqref{1radtr} has a curve of Morse index one solutions $V_\a$ for $\a\in I\subset (0,+\infty)$, then problem \eqref{1} has a curve of radial solutions $u_\a$, whose  Morse index changes when $\alpha$ is an even integer and goes to $+\infty$ as $\a\to + \infty$.
This  change in the Morse index yields nonradial solutions, by using  topological methods based on Leray-Schauder degree theory, as in Krasnoselski \cite{K64} or in the global bifurcation result of Rabinowitz \cite{R}.
A crucial point in doing that is the choice of the functional space: it is needed a compact operator in a Banach space that contains the radial solutions. 
Compactness does not hold trivially here, because we are working in an unbounded domain  and  with a possibly supercritical  nonlinearity. It can be recovered by picking a suitable weighted space that may vary depending on the nonlinear term $F$. To go further,  we need to formulate some assumptions on $F$. In any case, a similar bifurcation result could be obtained for other types of nonlinearities, by choosing another suitable functional space. 
Here we assume that   
\begin{equation}
\label{Fin0}
F(0)=0\,,\,F'(0)=-m< 0.
\end{equation}
%
\remove{
Let us give some definitions. For every $g\in L^{\infty}(\R^{\n})$ we define the weighted norm
\begin{equation}\label{norm}
\| g\|_{*}: =\sup _{x\in \R^{\n}}e^{|x|}|g(x)|
\end{equation} 
and we define $L^{\infty}_{*}(\R^{\n}):=\{g\in L^{\infty}(\R^{\n}) \hbox{ such that } e^{|x|}g(x)\in L^{\infty}(\R^{\n}) \}.$ Finally we let
\begin{equation}\label{X} X =D^{1,2}(\R^{\n}) \cap L^{\infty}_*,
\end{equation}
where $D^{1,2}(\R^{\n})$ is the completion of $L^{2^*}(\R^{\n})$ under the norm $\Arrowvert v\Arrowvert _{1,2}=\left( \int_{\R^{\n}}|\nabla v|^2\, dx \right)^{\frac 12} $. 
$X$ is a Banach space.
We consider on $X$ the following  norm
\[ \| v\|_{X}=\max\left\{ \|\nabla v\|_{L^2} \, , \, \| v\|_{*}\right\} .\]
}%
In this way, the  solution to \eqref{1radtr} has more than exponential decay at infinity (see Lemma \ref{morethanexpV}), and the radial solution $u_{\a}$ to \eqref{1} belong to
\begin{equation}\label{X} X = \{ v\in D^{1,2}(\R^{\n}) \, : \, e^{|x|} v(x) \in L^{\infty}(\R^{\n})\},
\end{equation}
where $D^{1,2}(\R^{\n})$ is the completion of $L^{2^*}(\R^{\n})$ under the norm $\Arrowvert v\Arrowvert _{1,2}=\left( \int_{\R^{\n}}|\nabla v|^2\, dx \right)^{\frac 12} $. 
It is clear that $X$, endowed with the norm
\[ \| v\|_{X}=\max\left\{ \|\nabla v\|_{L^2} \, , \, \| e^{|x|} v(x)\|_{\infty}\right\} ,\]
is a Banach space. 
Our main result reads as
\begin{theorem} \label{teo1}
Assume that \eqref{Fin0} holds, and that the problem \eqref{1radtr} with the boundary conditions \eqref{bc} has a  nondegenerate, Morse index one, solution  $V_\a$ 
for any $\a$ in an open interval $I\subset (0,+\infty)$. 
Let $\a_i=2i\in I$ with $i\in\mathbb{N}$. Then, \\
i) If $i$ is even, there exists at least a continuum of nonradial solutions to \eqref{1}, $O(\n-1)$-invariant, bifurcating from the pair $(\a_i,u_{\a_i})$ in $I\times X$.\\
ii) If $i$ is odd,  there exist at least $\big[\frac {\n}2\big]$ continua of nonradial solutions to
\eqref{1} bifurcating from the pair $(\a_i,u_{\a_i})$ in $I\times X$. Each branch is invariant w.r.t.~the action of the group $O(h)\times O(\n-h)$ for $h=1,2,\dots  \left[\frac {\n}2\right]$.\\
Moreover all these solutions have exponential decay, in the sense that
\[\limsup _{|x|\to +\infty} e^{|x|}v(x)<+\infty.\]
Finally the bifurcation is global and the Rabinowitz alternative holds.
\end{theorem}

Theorem \ref{teo1} implies that a branch of nonradial bifurcating solutions either is unbounded in $I \times X$, or  it meets the boundary $\partial I \times X$, or it connects to another branch. 
The last two occurrences  give a multiplicity result for solutions to \eqref{1}. 
We believe that, at least for some values of $\a$, these branches live for fixed $\alpha$ and are unbounded, see \cite{GM}. 

Theorem \ref{teo1} shows that the structure of solutions to \eqref{1} is much more complex than the case $\a=0$ and the presence of the term $|x|^{\a}$ produces nonradial solutions. These solutions arise  when $\a$ is an even number, whatever the nonlinear term is.
So far we state our results for problem \eqref{1} assuming that the dependence on the parameter $\a$ is only in the radial term $|x|^{\a}$. But we can allow the nonlinearity $F$ to depend on $\a$ too. Indeed the results of Propositions \ref{p:radial}, \ref{lin} and Corollary \ref{cor1} follow exactly in the same way since the transformed problem \eqref{1radtr} already depends on $\a$. The nonradial bifurcation result in Theorem \ref{teo1} can be extended to this case assuming, instead of \eqref{Fin0} that $F$ is $C^1(0,\infty)$ with respect to the parameter $\a$ and satisfies $F(\a,0)=0$ and $D_uF(\a,0)=-m_{\a}\leq -m<0$ for any $\a\in I$ and substituting in the proof $m$ with $m_\a$.
Some more knowledge of the function $F$ is actually needed to investigate the Morse index and the degeneracy of $V_{\a}$, in order to apply Theorem \ref{1}. 
In Section \ref{s4} we give some examples of effective nonlinearities  to which our bifurcation result applies, and deduce existence of nonradial solutions in subcritical and supercritical settings.
The rest of the paper is devoted to proofs. Section \ref{s2} is focused on radial solutions via problem \eqref{L}, and contains the proofs of  Propositions \ref{p:radial}, \ref{lin} and Corollary \ref{cor1}. The bifurcation result  Theorem \ref{teo1} is proved in Section \ref{s3}. Some more technical details are postponed to the Appendix.

\section{Some applications}\label{s4}

The statements of Propositions \ref{p:radial}, \ref{lin} and Theorem \ref{teo1}  are quite general, and they do not rely upon restrictive assumptions on the nonlinearity $F$.
In the following we single out some settings where they can actually be applied. 

\subsection{The case $F(u)=u^p-u$} 
In \cite{K}, Kwong, proved that problem \eqref{1radtr}, \eqref{bc} has a unique positive solution $V_\a$ for every $1<p<\frac{k+3}{k-1}$, i.e.  for every $\a > \max \{0,\frac{(\n-2)p-(\n+2)}{2}\}$, since $k(\a)=\frac{2\n-2+\a}{2+\a}$.
Then Proposition \ref{p:radial} yields that the problem
\begin{equation}\label{modello}
\left\{\begin{array}{ll}
-\Delta u=|x|^{\a}(u^p-u) & \hbox{ in }\R^{\n}\\
u>0& \hbox{ in }\R^{\n}\\
u\to 0&\hbox{ as }|x|\to +\infty
\end{array}\right.
\end{equation}
with $\n\geq 3$ and $p>1$
has unique radial solution $u_\a$ for every $\a> \max \{0,\frac{(\n-2)p-(\n+2)}{2}\}$. Further, $V_{\a}$ and $V_{\a}'$ have exponential decay so that belong to space $E$ (see Lemma \ref{morethanexpV}). These solutions $V_\a$ can be found by the Mountain Pass Theorem and so their Morse index is at most one. Let us consider the linearized equation of \eqref{1radtr} at this solution $V_\a$, i.e. 
\begin{equation}\label{lin-modello}
-v''-\frac ktv'=(pV_\a^{p-1}-1)v
\end{equation}
where $v\in E$.
It can be easily seen, differentiating equation \eqref{1radtr}, that the first eigenvalue of the linearized equation \eqref{lin-modello} is $-k<0$, so that $V_\a$ is a Morse index one solution to \eqref{1radtr} and \eqref{bc} for any $\a>\max \{0,\frac{(\n-2)p-(\n+2)}2\}$.\\ 
Moreover the second eigenvalue is strictly positive since it has be shown in \cite{K} that all the solutions to \eqref{lin-modello} are unbounded. This implies that $V_\a$ is nondegenerate for any admissible $\a$ and has Morse index one.\\
In the subcritical or critical case ($1<p\le\frac{\n+2}{\n-2}$ fixed), Theorem \ref{teo1} implies that problem \eqref{modello} has a curve of radial solution $u_{\a}$ for every $\a\geq 0$, and that for every {\em even } $\a$ the point $(\a, u_{\a})$ is a nonradial bifurcation point and a global branch of nonradial solutions exists. When $\a$ is even but is not divisible by $4$ we find $\left[\frac {\n}2\right]$ different global branches of nonradial solutions bifurcating from $(\a, u_{\a})$.
In the subcritical case the existence of the radial solution $u_{\a}$ is a standard result, but we are able to prove existence of nonradial solutions and multiplicity results showing that the structure of solutions to \eqref{modello} is much more complex than the case $\a=0$. When the exponent $p$ is critical also the existence of the radial solution $u_{\a}$ is a new result.\\ 
When the exponent $p$ is supercritical ($p>\frac{\n+2}{\n-2}$ fixed), then problem \eqref{modello} has a curve of radial solution $u_{\a}$ for every $\a> \frac{(\n-2)p-(\n+2)}{2}$. Again we can apply Theorem \ref{teo1}
getting one global branch of nonradial solutions that bifurcate from  the radial solution $(\a, u_{\a})$ when  $\a>\frac{(\n-2)p-(\n+2)}{2}$ is an even number
and $\left[\frac {\n}2\right]$ different global branches of nonradial solutions
 when $\a$ is even but not divisible by $4$.
Thus the transformation \eqref{cov} and Theorem \ref{teo1} yield existence of both radial and nonradial solutions for \eqref{modello} when $p>1$ is any number and $\a$ is large enough depending on $p$.
Existence results with supercritical exponents are usually very difficult to prove.


\subsection{The case  $F(u)=u^{\frac{N +2+2\a-\e}{N-2}}-u$}
It is possible to let the nonlinear term $F$ depend on $\a$. For instance, we can take  $F(u)=u^{\frac{\n +2+2\a-\e}{\n-2}}-u$, where $\e$ is a fixed positive number (such that $\e<4$) so that $F(u)$ 
is supercritical when $\a>\frac{\e}2$. We can apply again the existence results of Kwong, see \cite{K}, getting that problem \eqref{1radtr}, \eqref{bc} has a unique positive solution $V_\a$ for every $\a\geq 0$.\\
As before $V_\a$ is a Morse index one, nondegenerate solution that belongs to the space $E$. Then  Proposition \ref{p:radial} implies that the supercritical problem
\begin{equation}\label{modello2}
\left\{\begin{array}{ll}
-\Delta u=|x|^{\a}(u^{\frac{\n +2+2\a-\e}{\n-2}} -u) & \hbox{ in }\R^{\n}\\
u>0& \hbox{ in }\R^{\n}\\
u\to 0&\hbox{ as }|x|\to +\infty
\end{array}\right.
\end{equation}
has unique radial solution $u_\a$ for every $\a> 0$.\\
We can apply Theorem \ref{teo1} and we get  a global branch of nonradial solutions that  bifurcate from $(\a, u_{\a})$ that when $\a$ is an even number and $\left[\frac {\n}2\right]$ different global branches of nonradial solutions
 when $\a$ is even but not divisible by $4$.\\
Again Proposition \ref{p:radial} and Theorem \ref{teo1} produce existence and multiplicity results for a supercritical nonlinear problem.

\subsection{Other nonlinearities}
The general bifurcation result applies also to a class of nonhomogeneous nonlinearities, which extends the case $F(u)=u^{p}-u$, and has been widely analyzed, mainly in the framework of Schrodinger equation or autonomous elliptic problems like \eqref{BLauto} (for instance \cite{BL-ARMA}).
In addition to \eqref{Fin0}, we take that $F$ increases like a power at infinity:
\begin{align}
\label{Fcresce}
\limsup\limits_{u\to+\infty}\frac{| F(u)|}{|u|^p} & = \ell < \infty ,
\end{align}
for some $p>1$, and that 
\begin{align}
\label{Fnozero}
\mbox{there exists $s>0$ such that} \int_0^{s} F(u) du >0.
\end{align}
To ensure uniqueness of the radial solution, we choose the setting of \cite{KZ91}, i.e. 
\begin{align}
\label{F2}
\begin{array}{c} \mbox{there exists $\theta >0$ such that} \; F(u)  < 0  \mbox{ for } u <\theta  \mbox{ and } F(u)>\theta  \mbox{ for } u >\theta 
\\
\mbox{ and } F^\prime(u) \le 0  \mbox{ in a neighborhood of } u=0.
\end{array}
\\
\label{F3} 
\begin{array}{c} \mbox{Let $\phi$ such that } \int_0^\phi F(u) du =0, \mbox{and } G(u)= u F^\prime(u)/F(u):
\\
\mbox{in } [\phi,+\infty), \; G(u) \mbox{ is nonincreasing and there exists } \lim\limits_{u\to+\infty} G(u)=\lambda \ge 1,
\\
\mbox{in } [\theta,\phi), \; G(u) \ge G(\phi) , \; \mbox{ and in } [0,\theta), \; G(u) \le \lambda.
\end{array} 
\end{align}
Slightly different hypotheses should guarantee uniqueness as well (see \cite{McL93}). In any case, we are able to include nonlinearities of ``polynomial'' type like
\[F(u) = u^p + \sum\limits_{1< q<p} c_q u^q -m\,u,\]
with $p>1$, $m>0$, and some restriction on the coefficients $c_q$ (see \cite[Section 4]{KZ91}).
Other functions that match assumptions \eqref{Fin0}, \eqref{Fcresce}--\eqref{F3} are 
\begin{align*}
F(u) = & \max\{ u^p , u^q\} -m\,u ,
\intertext{ with $p>q>1$ and $m>0$, and}
F(u) = & \dfrac{u^q}{1+u^{q-p}} - m\,u,
\end{align*} 
at least for $p>q\ge (p+3\sqrt{m})/(1+3\sqrt{m})>1$ and $m>0$.

In this setting all the  hypotheses of Theorem \ref{teo1} are satisfied.
First, problem \eqref{1radtr}, \eqref{bc} has an unique positive radial solution for all values of $\alpha > \max\left\{ 0 \, , \, \frac{(\n-2) p-\n-2}{2} \right\}$. When $k(\a)$ is an integer, a solution  exists by the well known results for the autonomous elliptic problem \eqref{BLauto}, relying on a constrained minimization method (see \cite[Theorem 1]{BL-ARMA}). Their arguments need some refinements to handle the general case $k(\a)>1$, and thus provide a continuum of solutions (paramerized by $\alpha$).

\begin{proposition}\label{p2.2}
Under assumptions \eqref{Fin0}, \eqref{Fcresce}--\eqref{F3}  problem \eqref{1radtr}, \eqref{bc} has an unique positive solution $V_\a$ for any $\a > \max\left\{ 0 \, , \, \frac{(\n-2) p-\n -2}{2} \right\}$.
\end{proposition}
A  sketch of the proof is reported in the Appendix.

\begin{remark}\label{morseetc}
By  Lemma \ref{morethanexpV},  $V_{\a} , V_{\a}'\in E$.
On the other hand the  Morse index  of  $V_\a$ is  equal to one, because it has been produced as  a minimum, constrained on a manifold with co-dimension one (see \cite[Remark 2.12]{E}). Actually it is easily seen that $V_{\a}'$ is an eigenfunction, and that the value of the first eigenvalue of \eqref{Larmtr} is $-k(\a)<0$.
Next, an eventual  nontrivial solution  to \eqref{Larmtr} corresponding to $\lambda=0$ should have only one zero. Then it could not vanish  at infinity by \cite[Lemma 9]{KZ91}. For this reason the solution $V_{\a}$ is nondegenerate in $E$.
\end{remark}

Now Proposition \ref{p:radial} ensures that the problem \eqref{1} has unique radial solution $u_\a$ for every $\a> \max \{0,\frac{(\n-2)p-(\n+2)}{2}\}$, and that such solution is in the space $X$. Moreover Proposition \ref{lin}\emph{(i)}, Corollary \ref{cor1} and Theorem \ref{teo1} apply. In particular, when the parameter $p$ appearing in \eqref{Fcresce} is supercritical (i.e.~$p>\frac{\n+2}{\n-2}$), we get the well-posedness of problem \eqref{1} in the class of positive radial solutions for any $\alpha>\frac{(\n-2)p-(\n+2)}{2}$, and existence of branches of nonradial solutions that bifurcate at any even value of $\alpha$.

\section{Radial solutions}\label{s2}

We focus here on radial solutions, and exploit the links between radial solutions to \eqref{1} and solutions to \eqref{1radtr}--\eqref{bc}.
For the sake of completeness, we begin by proving Proposition \ref{p:radial}.

\begin{proof}[Proof of Proposition \ref{p:radial}] 
It is clear that, if \eqref{1radtr} has a solution $V_\a$ that satisfies \eqref{bc}, then \eqref{1} has the radial solution $u_\a$ defined in \eqref{ualfa}.
Viceversa, let $u_\a$ be a radial solution of \eqref{1}. Then the function $V_\a(t)=u_\a(r)$, where $t$ and $r$ are related by \eqref{cov}, satisfies \eqref{1radtr} with the boundary conditions
\begin{equation}\label{bc0}\nonumber
\lim_{t\to 0^+}t^{\frac{\a}{2+\a}}V_\a'(t)=0\,,\quad \lim\limits_{t\to+\infty}V_\a(t) =0.
\end{equation}
Then
$\lim_{t\to 0^+}t^{k}V_\a'(t)=0$ since $k-\frac{\a}{2+\a}>0$ by the definition of $k$. Integrating \eqref{1radtr} we have
$$-t^kV_\a'(t)=\int_0^t s^k F(V_\a(s))\, ds$$
and this implies that 
$$\lim_{t\to 0^+}V_\a'(t)=-\lim_{t\to 0^+}\frac{ \int_0^t s^k F(V_\a(s))\, ds}{t^k}=0.$$ 
Therefore $V_\a$ is a solution to \eqref{1radtr}--\eqref{bc}.
\end{proof}

We next address to the linearization of problem \eqref{1} around the radial solution $u_{\a}$ produced in Proposition \ref{p:radial}, namely we study problem \eqref{L} and  prove Proposition \ref{lin} and Corollary \ref{cor1}.
\begin{proof}[Proof of Proposition \ref{lin}]
We investigate  the degeneracy of $u_{\a}$  by decomposing an eventual solution $w$ to \eqref{L} according to the spherical harmonics $Y_i(\theta)$ and write
\[w(x)=\sum\limits_{i=0}^{\infty} w_i(|x|) Y_i(\theta).\]
Now  $w$ is a nontrivial solution to \eqref{L} if and only if any non-null $w_i$ has the summability
\[\int _0^{+\infty}r^{\n-1}(w_i')^2+\int _0^{+\infty} r^{\n-1+\a}w_i^2\, dr <\infty\]
and solves
\begin{equation}\label{Larm}
- w^{\prime\prime}_i - \dfrac{\n-1}{r}w^{\prime}_i + \dfrac{\mu_i}{r^2} w_i = r^{\alpha} F'(u_{\alpha}) w_i , 
\end{equation}
where $\mu_i=i(\n-2+i)$ is the $i^{th}$ eigenvalue of the Laplace Beltrami operator on the $(\n-1)$-dimensional sphere and has multiplicity $\n_i=\frac{(\n+2i-2)(\n+i-3)!}{(\n-2)!i!}$, see for example \cite{SW}.
The change of variable \eqref{cov} transforms problem \eqref{Larm} into
\[
- v^{\prime\prime}_i - \dfrac{k}{t}v^{\prime}_i = F'(V_{\alpha}) v_i - \dfrac{4\mu_i}{(2+\alpha)^2 t^2} v_i 
\]
for $v_i\in E$.
Hence $u_{\alpha}$ is degenerate if and only if $\lambda=-4\mu_i/(2+\alpha)^2\le 0$ is a nonpositive eigenvalue for the singular weighted eigenvalue problem 
\begin{equation}\label{Larmtr}
- v^{\prime\prime}- \dfrac{k}{t}v^{\prime} = F'(V_{\a}) v +\dfrac{\lambda}{t^2} v\qquad t>0 
\end{equation}
in $E$. Lemma \ref{x1} in the Appendix gives that the singular weighted eigenvalue problem \eqref{Larmtr} has only one negative eigenvalue $\l<0$ with eigenfunction in $E$.
Let us check that its value is $-k$: indeed, deriving equation \eqref{1radtr} w.r.t.~$t$ gives that
$W=V'_{\alpha}$ satisfies
\[
- W^{\prime\prime}- \dfrac{k}{t} W^{\prime}= F'(V_{\alpha}) W- \dfrac{k}{t^2} W \qquad t>0 
\]
and $W\in E$ by assumption.
So, if $0$ is not an eigenvalue for \eqref{Larmtr}, the linearized problem \eqref{L} has nontrivial solution if and only if there exists $i$ such that $4\mu_i/(2+\alpha)^2=k$, or equivalently
\[ i(\n-2+i) = \left(1+\frac{\alpha}{2}\right) \left(\n-1+\frac{\alpha}{2}\right), \]
which means that $\alpha=2(i-1)$ for some integer $i\ge 2$. 
Otherwise, if also $0$ is an eigenvalue for \eqref{Larmtr}, another type of nontrivial solutions to \eqref{L} shows up: the radial ones coming from $\mu_i=0$ i.e. the functions  $w_0(x)$. The presence of these radial solutions does not depend by the value of $\alpha$.
\end{proof}
\begin{proof}[Proof of Corollary \ref{cor1}]
Lemma \ref{x2} in the Appendix allows to compute the Morse index of $u_\a$ by counting  the negative eigenvalues with weight:
\begin{equation}\label{LA}
-\Delta w=|x|^{\a} F'(u_{\a}) \, w + \dfrac{\lambda}{|x|^2} w \qquad   \hbox{ in }\R^{\n} 
\end{equation}
with the summability condition $|x|^{\alpha/2} w, |\nabla w| \in L^2(\R^{\n})$.
Using as before the decomposition along spherical harmonics and the change of variable \eqref{cov} we have that 
every negative  eigenvalue of \eqref{LA} has to be of type $\lambda= \frac{4\mu_i}{(2+\alpha)^2}-k(\a)$, and that its multiplicity is $\n_i$.
Hence the Morse index of $u_{\alpha}$ can be computed by counting all index $i$ with $\frac{4\mu_i}{(2+\alpha)^2}< k(\a)$, each with its multiplicity. The thesis follows by recalling the formula for $k(\a)$. 
\end{proof}
\begin{remark}\label{r1}
{\rm Reasoning as in the proof of Proposition \ref{lin} and Corollary \ref{cor1} it is easy to see that any eigenfunction of \eqref{LA} corresponding to a negative eigenvalue $\l$ can be written in the following way
$$w(x)= V_\a'(|x|)Y_{j,h}(\theta)$$
where $Y_{j,h}$ are the spherical harmonics related to the $j$-th eigenvalue. }
\end{remark}

\section{Proof of the bifurcation result}\label{s3}
\remove{
It is clear that $X$, endowed with the norm
\[ \| v\|_{X}=\max\left\{ \|\nabla v\|_{L^2} \, , \, \| e^{|x|} v(x)\|_{\infty}\right\} ,\]
is a Banach space. 
\\
We recall
\begin{definition} \label{def1}
We say that a nonradial bifurcation occurs at $(\overline\a,u_{\overline{\a}})$ if in every neighborhood of $(\overline\a,u_{\overline{\a}}) $ in $(0,+\infty)\times X$ there exists a point  $(\a,v_\a)$ with $v_\a$ nonradial solution of  \eqref{1}.
\end{definition}
}
In this section we will prove Theorem \ref{teo1}.
Under assumption  \eqref{Fin0}, any solution $V_{\a}$ of \eqref{1radtr}, \eqref{bc} have an exponential decay at infinity, together with its first derivative $V_{\a}'$(see Lemma \ref{morethanexpV} in the Appendix). This implies that the functions $V_{\a}$ and $V'_{\a}$ belong to the space  $E$.
Then a curve of solutions $V_{\a}$ (for $\alpha\in I$) gives rise to a curve of radial solutions $u_\a\in X$ for problem \eqref{1} \eqref{bc1}, and any $u_\a$ is radially nondegenerate, provided that  $V_\a$ is nondegenerate, thanks to Propositions \ref{p:radial} and \ref{lin}.
Let $\a_i=2i$ with $i\in \N$. We want to prove that, if $\a_i\in I$ for some $i\in \N$ then $(\a_i, u_{\a_i})$ is a nonradial bifurcation point. To this end we extend $F$ to a function from $\R\to \R$ in an odd way, introduce a new function 
\[  f (s)= F(s)+m \, s , \quad s\in \R,\]
and define the operator 
\[ T: I\times X \to X , \qquad T(\alpha,v)=\left(-\Delta +m |x|^{\alpha}I\right)^{-1} (|x|^{\alpha}f(v)).\]
We recall that the parameter $m$ and the set $X$  have been defined, respectively,  in \eqref{Fin0} and \eqref{X}.
In that way,  $T(\alpha,v)=z$ if and only if $z$ solves
\begin{equation}
\label{Teq} \left\{\begin{array}{ll}
-\Delta z +m|x|^{\alpha}z = |x|^{\alpha}f(v) & \text{ in } \R^{\n}\\
z\to 0 & \text{ as } |x|\to \infty
\end{array}\right.
\end{equation}
and a solution $v$ of problem \eqref{1} and \eqref{bc1} is a fixed point for the operator $T$, i.e it satisfies $T(\alpha,v)=v$. \\
Let us check, first, that the operator $T(\alpha,\cdot)$ is well defined and compact, for fixed $\alpha>0$.

\begin{lemma}\label{Tdef} Assume \eqref{Fin0}.
For all $\alpha>0$, $T(\alpha,\cdot)$ is a compact operator on $X$.
\end{lemma}
\begin{proof}
We begin by checking that the operator $T$ is well defined on $X$, i.e.~that equation \eqref{Teq} is well-posed.
A solution can be produced by a classical exhaustion argument: 
take $z_n$ the solution to the standard elliptic Dirichlet problem in a ball of radius $n$
\begin{equation}\label{eq}
\left\{\begin{array}{ll}
-\Delta z_n +m|x|^{\alpha}z_n = |x|^{\alpha}f(v), \qquad & \text{in } B_n , \\
z_n=0 , & \text{on } \partial B_n.
\end{array}\right.
\end{equation}
We undertake that $z_n=0$ outside $B_n$. 
Let us denote $\bsc=\|v\|_X$ and  take  $g$ a modulus of continuity for $f(u)/u$ on $[-\bsc,\bsc]$, i.e.~a continuous and nondecreasing function with $g(0)=0$ and
\[ |f(u)| \le g(|u|) \, |u| \qquad \mbox{ for all } |u|\le\bsc.\]
Such function $g$ exists in virtue of assumption  \eqref{Fin0}.
Next, let $\zeta_n$ be the positive radial solution to
\[
\left\{\begin{array}{ll}
-\Delta \zeta_n +m|x|^{\alpha}\zeta_n = \bsc  |x|^{\alpha} e^{-|x|} g\left(\bsc e^{-|x|}\right), \qquad & \text{in } B_n ,\\
\zeta_n=0 , & \text{on } \partial B_n.
\end{array}\right.
\]
Comparison principle yields  that $ |z_n|\le \zeta_n$.
To obtain uniform estimates  at infinity, we consider the radial positive solution to the global problem
\begin{equation}\label{eq:z}
\left\{\begin{array}{ll}
-\Delta \zeta +m|x|^{\alpha}\zeta =  \bsc  |x|^{\alpha} e^{-|x|} g\left(\bsc e^{-|x|}\right) , \qquad & \text{in } \R^{\n} ,\\
\zeta(x)\to 0 , & \text{as } |x|\to\infty.
\end{array}\right.
\end{equation}
It is clear (by comparison) that $\zeta_n\le\zeta$ and therefore by Lemma \ref{piucheesp}
\begin{equation}\label{dec-exp}
 |z_n|\le \const e^{-|x|},
\end{equation} 
where the constant $\const$ does not depend on $n$.
Estimate \eqref{dec-exp} implies, using \eqref{eq} that
 $z_n$ is uniformly bounded in $ {\mathcal D}^{1,2}(\R^{\n})$ and converges  weakly in ${\mathcal D}^{1,2}(\R^{\n})$ to a weak solution $z$ of \eqref{Teq} with the regularity  $|x|^{\frac \alpha 2} z, |\nabla z| \in L^2$.
Since $z_n$ converges to $z$ pointwise a.e. (up to a subsequence) we get that $|z|\le \const e^{-|x|}$
so that $z\in X$ as requested.
Lastly, uniqueness follows by standard energy estimates.

To check compactness, let us take $v_n$ a bounded sequence in $X$: in particular  $v_n$ converge towards some $v\in X$ weakly in ${\mathcal D}^{1,2}$, strongly in any $L^q$ (as $q>1$) and pointwise a.e. (up to a subsequence). Let then  $z_n=T(\alpha,v_n)$.
The same arguments as before imply that $z_n$ are uniformly bounded in $X$.
Hence also the sequence $z_n$ converges towards some $z$ weakly in ${\mathcal D}^{1,2}$, strongly in any $L^q$ (as $q>1$) and pointwise a.e.  (up to a subsequence). Moreover it is easy to check that $z$ is a weak solution to \eqref{Teq}. Therefore 
\begin{align*}
\int_{\R^{\n}} |\nabla z_n|^2 dx = \int_{\R^{\n}}|x|^{\alpha}f(v_n)  z_n dx - m \int_{\R^{\n}} |x|^{\alpha} |z_n|^2 dx \\
\to \int_{\R^{\n}}|x|^{\alpha}f(v)  z dx - m \int_{\R^{\n}} |x|^{\alpha} |z|^2 dx = \int_{\R^{\n}} |\nabla z|^2 dx  
\end{align*}
as $n\to\infty$.
So $\nabla z_n \to \nabla z$ strongly in $ L^2(\R^{\n})$.
Besides, the same arguments of the proof of Lemma \ref{Tdef} yields that $|z| , |z_n| \le \zeta$, where $\zeta$ is the solution to \eqref{eq:z} with $\bsc\ge \|v_n\|_X$ for all $n$. We thus  compute
\begin{align*}
\sup\limits_{\R^{\n}} |z_n-z| e^{|x|} \le \sup\limits_{B_R} |z_n-z| e^{|x|} + \sup\limits_{\R^{\n}\setminus B_R} |z_n-z| e^{|x|} \\
\le \sup\limits_{B_R} |z_n-z| e^{|x|} + 2 \sup\limits_{\R^{\n}\setminus B_R} |\zeta| e^{|x|} .
\end{align*}  
The last term vanishes as $R\to\infty$ in virtue of Lemma \ref{piucheesp}. Eventually $z\in X$ and the thesis follows, because $z_n\to z$ locally uniformly. 
\end{proof}

Next Lemma inherits $T'_v$, the Fr\'echet derivative of the operator $T$ with respect to $v\in X$.

\begin{lemma}\label{Tinvert} Assume \eqref{Fin0}, and take that problem \eqref{1radtr}, \eqref{bc} has a nondegenerate solution $V_{\alpha}$ for every $\a\in I$.
Then the linearized operator $T'_v(\a,u_\a)$ is invertible for every $\a\in I$ such that $\a\neq \a_i$.
\end{lemma}
This fact is a byproduct of Proposition \ref{lin}, and we omit the proof.

Let us introduce some notations concerning the symmetries of $\R^N$:
\[H:=\{v\in X\, :\, v(x_1,\dots,x_{\n})=v(g(x_1,\dots x_{\n-1}),x_{\n})\, \hbox{ for any } g \in O(\n-1)\},\]
and the subgroups of $O(\n)$:
\[G_h=O(h)\times O(\n-h) \quad \qquad \hbox{ for }1\leq h\leq \left[\frac {\n}2\right]\]
where $[a]$ stands for the integer part of $a$. Also, we  denote by $H^h$ the subspaces of $X$
 of functions invariant by the action of $G_h$.\\
The result of Smoller and Wasserman in \cite[Proposition 5.2]{SW} implies that, for $j=0,1,\dots$ the eigenspace of the Laplace Beltrami operator related to $\mu_j$, i.e. the solutions to $-\Delta_{S^{\n-1}}Y_j=\mu_jY_j$ contains only one eigenfunction which is $O(N-1)$-invariant. The same authors in \cite[Lemma 6.5]{SW90} prove that when $j$ is even the eigenspace of the Laplace Beltrami operator related to $\mu_j$ has only one solution which is invariant by the action of $G_h$
for $h=1,\dots,\left[\frac {\n}2\right]$.
Then Corollary \ref{cor1} and Remark \ref{r1}  imply that
\begin{equation}\label{cambio}
m_{H}(\a_i+\e)-m_{H}(\a_i-\e)=1
\end{equation}
if $\e$ is small enough, where $m_{H}$ denotes the Morse index of $u_\a$ in the space $H$ (or $H^h$). This odd change in the Morse index is responsible of the bifurcation.\\
First we prove the local bifurcation result.
\begin{proposition}\label{local}
Under the same hypotheses of Theorem \ref{teo1}, the points $(\a_i,u_{\a_i})$ are nonradial bifurcation points for the curve $(\a,u_\a)$ in the space $I\times H$ (or $I\times H^h$).
\end{proposition}
\begin{proof} 
Assume by contradiction that  $(\a_i,u_{\a_i})\in I\times X$ is not a bifurcation point for \eqref{1} for some $i$. Then there exists $\e_0>0$ such that $\forall \e\in (0,\e_0)$ and $\forall c\in (0,\e_0)$ 
$$v-T(\a,v)\neq 0$$
for any $\a\in (\a_i-\e,\a_i+\e)\subset I$ and for any $v\in H$ (or in $H^h$) such that $\nor v-u_\a\nor_X\leq c$ and $v\neq u_\a$.\\
Observe that for every $\a$ the pair $(\a,u_{\a})$ is a solution of \eqref{1} and this implies that satisfies $u_\a-T(\a,u_\a)=0$ for every $\a\in I$. From Lemma \ref{Tdef} we know that the Leray Schauder degree of $I-T$ is well defined in $X$ and hence also in $H$ (or in $H^h$).\\
Now we consider the case of the space $H$.
Let $\Gamma:=\{(\a,v)\in (\a_i-\e,\a_i+\e)\times H\,\,:\,\,\nor v-u_\a\nor_X\leq c\}$ and $\Gamma_\a:=\{v\in H\hbox{ s.t. }(\a,v)\in \Gamma\}$. By the homotopy invariance of the Leray Schauder degree
we have that
\begin{equation}\label{constant}
deg(I-T(\a,\cdot), \Gamma_\a,0) \hbox{ is constant on }(\a_i-\e,\a_i+\e).
\end{equation}
As proved in \cite[Theorem 3.20]{AM} the Leray Schauder degree in \eqref{constant} for
$\a\neq \a_i$ is equal to $(-1)^{\gamma}$ where $\g$ is the number of the eigenvalues
of $T'_v(\a,u_{\a})$  counted with multiplicity contained in $(1,+\infty)$.
We know that $\L$ is  
an eigenvalue for the linear operator $T'_v(\a,u_{\a})$  if and only if 
$$\L I-T'_v(\a,u_{\a})I=0$$
has a nontrivial solution.   
This means we have  to find $w\in H$, $w\neq 0$ which verifies
\begin{equation}\label{4.5}
-\Delta w+ m|x|^{\a}w=\frac 1{\L}|x|^{\a}f'(u_\a)w \quad \hbox{ in }\R^{\n}
\end{equation}
for some $\frac 1{\L}\in (0,1)$. We can infer then that the Leray Schauder degree in \eqref{constant}  for
$\a\neq \a_i$ is equal to $(-1)^{m(\a)}$ where $m(\a)$ is the Morse index of $u_\a$ in the space $H$.\\
Then we have 
$$deg(I-T(\a_i\pm \e,\cdot), \Gamma_{\a_i\pm\e},0)=(-1)^{m(\a_i\pm\e)}$$
so that \eqref{cambio} implies 
$$deg(I-T(\a_i- \e,\cdot), \Gamma_{\a_i-\e},0)=-deg(I-T(\a_i+ \e,\cdot), \Gamma_{\a_i+\e},0)$$
contradicting \eqref{constant}. Then $(\a_i,u_{\a_i})$ is a bifurcation point for \eqref{1} and the
bifurcating solutions are nonradial since $u_\a$ is radially nondegenerate via Proposition \ref{lin},
 because by assumption $V_\a$ is nondegenerate.
\end{proof}

Eventually, we complete the proof of the global bifurcation result.

\begin{proof}[Proof of Theorem \ref{teo1}]
The  global bifurcation follows by Lemmas \ref{Tdef}, \ref{Tinvert} and Proposition \ref{local} by using the Leray-Schauder degree
theory, as in the classical Rabinowitz result \cite{R}; see also \cite{AM} for a proof.

The final step to complete the proof is to show that when $i$ is odd, then the bifurcating solutions we find in $H^h$ are distinct. Indeed, any solution $v$ which is invariant for the action of two distinct groups $G_{h_1}$ and $G_{h_2}$ with $h_1\neq h_2$ should be radial (see \cite[Lemma 6.3]{SW90}), and this is not possible because $u_\a$ is radially nondegenerate by Proposition \ref{lin}. 
\end{proof}

\begin{remark}\rm 
Theorem \ref{teo1} states that the points $(\a_i,u_{\a_i})$ are bifurcation points, and that the bifurcation is global. Moreover the Rabinowitz alternative holds for the branches of nonradial solutions that bifurcate from $(\a_i,u_{\a_i})$. Hence we have that a branch of bifurcating solutions or it is unbounded in $I\times X$, or it meets $\partial I\times X$ or there exists $\a_j$ with $j\neq i$ such that $(\a_j,u_{\a_j})$ belongs to the same branch.
\end{remark}

\section{Appendix}

We report here for the sake of completeness some facts about ODE that have been used through the paper.
The first results show the link between the eigenvalue problem associated to the linearized equation \eqref{L}  and the eigenvalue problem with weight \eqref{Larmtr}. This weighted eigenvalue problem is required in the proof of Proposition \ref{lin} and its implications.

\begin{lemma}\label{x1}
Let $V_\a$ be a solution to \eqref{1radtr}. If $V_\a$ has Morse index one then the weighted eigenvalues problem \eqref{Larmtr} admits only one negative eigenvalue in $E$.
\end{lemma}
\begin{proof}
By definition of Morse index we have that 
\[ \inf_{\begin{subarray}{c}w\in E,\\ w\neq 0\end{subarray}}\frac{\int_0^{+\infty}r^k\left((w')^2-F'(V_{\a}) w^2\right) dr}{ \int_0^{+\infty}r^{k-2}w^2\, dr}=\Lambda_1(k)<0.\]
Arguing as in proof of  \cite[Proposition  5.4]{GG?} (see also  \cite[Proposition  A.1]{GGN2}), one can see that actually the minimum is attained, and  there exists a positive function $w_1\in E$ that solves \eqref{Larmtr} for $\lambda=\Lambda_1(k)$. 
Assume, by contradiction, there is another eigenfunction $w_2\neq 0\in E$ corresponding to another negative eigenvalue with weight $\Lambda_2(k)$. Then, using the equations satisfied by $w_1$ and $w_2$, we get $\int_0^{+\infty}r^{k-2}w_1w_2\, dr=0$ so that $w_1$ and $w_2$ are linearly independent. This implies that the quadratic form $\int_0^{+\infty}r^k\left((w')^2-F'(V_{\a}) w^2\right) dr$ is negatively defined on the two-dimensional subspace of $E$ spanned by $w_1$ and $w_2$ contradicting the fact that $V_\a$ has Morse index one.
\end{proof}

\begin{lemma}\label{x2}
Let $u_\a$ be the radial solution to \eqref{1}. Then the Morse index of $u_\a$ coincides with the number, counted with multiplicity, of the negative eigenvalues with
weight of the problem \eqref{LA}.
\end{lemma}
\begin{proof}
The proof follows as in \cite[Corollary 5.6]{GG?}: each time we have a negative eigenvalue, also the eigenvalue with weight \eqref{LA} is attained.
\end{proof} 

Afterwards, we recall some properties of the solution of the ordinary differential equation \eqref{1radtr}. Under the only assumption \eqref{Fin0}, every positive solutions that vanish at infinity have more than exponential decay, and the same holds for their derivatives; in particular, they belong  to the space $E$.

\begin{lemma}\label{morethanexpV}
Assume \eqref{Fin0}, and take $V$ a positive solution to \eqref{1radtr} vanishing at infinity.
 Then for every $\delta \in(0,\sqrt{m})$ there exist constants $\const$ and $\const_1$  such that
\[
V(t) \le \const t^{-\frac{k}{2}} e^{-\delta t} , \qquad |V'(t)| \le {\const}_1 t^{-\frac{k}{2}} e^{-\delta t} , \qquad 
\]
\end{lemma}

\begin{proof} We repeat here the arguments of \cite[Lemma 2]{BL-ARMA}.
Let $W(t)=t^{\frac{k}{2}}V(t)$, we have 
\[ W^{\prime\prime}= a\, W, \qquad \text{with } \ a(t)=\dfrac{k(k-2)}{4 t^ 2} - \dfrac{F(V(t))}{V(t)}.  
\]
Because $V(t)$ vanishes as $t\to+\infty$, we have $\liminf\limits_{t\to+\infty} a(t)=m>0$. Hence, for any $\delta\in(0,\sqrt{m})$, there is  $t_o$ such that $a(t)\ge \delta^2$, as $t\ge t_o$. Since $W(t)\ge 0$ we get
\[ W^{\prime\prime}\ge\delta^2 W, \qquad \text{as } \ t\ge t_o.  
\]
Next, let  $Z(t)=\left(\delta W(t)+W'(t) \right) e^{-\delta t}$. $Z$ is nondecreasing on $(t_o,+\infty)$ because $Z' = (W^{\prime\prime} - \delta^2 W)e^{-\delta t} \ge 0$. If there exist $t_1>t_o$ so that $Z(t_1)>0$, then $\delta W(t)+W'(t) \ge Z(t_1) e^{\delta t}$ as $t\ge t_1$. After computations we should obtain
\[
V'(t) + \left(\delta +\frac{k}{2t}\right) V(t)\ge Z(t_1)\, t^{-\frac{k}{2}}e^{\delta t},
\]
for all $t\ge t_1$, which contradicts the assumption $V(t)\to 0$ as $t\to+\infty$.
Hence $Z(t)\le 0$ in $(t_o,+\infty)$, indeed. This implies that
\[ \left(W e^{\delta t}\right)^{\prime}= Z e^ {2\delta t} \le 0  \qquad \text{ as } t\ge t_o,\]
and then $0\le W(t)\le \const e^{-\delta t}$, which gives  the thesis.

Concerning the first derivative,  deriving equation \eqref{1radtr} we have that $V'=P$ satisfies equation
\[
-P''-\frac kt P'=F'(V) P-\frac k{t^2} P \quad \hbox{in }(0,+\infty).
\]
Moreover, as in the proof of Proposition \ref{p:radial}
we have
$$P(t)=-t^{-k}\int_0^ts^kF(V(s))\, ds.$$
From \eqref{Fin0} and the exponential decay of $V$  the integral $\int_T^{+\infty}s^kF(V(s))\, ds$ is bounded,
so that $P$ vanishes at infinity. Thus arguments similar to the ones already used yield  that also the function $P$ has more than exponential decay at infinity and give the thesis.
\end{proof}

Next Lemma shows that any solution to \eqref{1radtr} is bounded at $t=0$.

\begin{lemma}\label{Vbd0}
Assume \eqref{Fin0} and \eqref{Fcresce}, and take $V$ a positive solution to \eqref{1radtr} with $p<\frac{k+3}{k-1}$.
If $V$ vanishes at infinity, then it is bounded near at $t= 0$.
\end{lemma}
\begin{proof}
We assume by contradiction that  $\lim\limits_{t\to 0} V(t) = +\infty$, and  perform a Kelvin transform
\[ \hat V (t) = t^{-k+1} V({1}/{t}) .\]
It is easily seen that
\begin{align}\label{eqkelvin}
-\left(t^k\hat{V}^{\prime}\right)^{\prime}= t^{-3} F(V({1}/{t})) \quad & & t>0 .
\end{align}
Moreover $t^k{\hat V}^{\prime}(t)= -(k-1)\,V({1}/{t}) -t^{-1}V'({1}/{t})\to 0$ as $t\to 0$ because $V(1/t)$ vanishes very fast (see Lemma \ref{morethanexpV}). Integrating \eqref{eqkelvin} gives
\begin{align*}
-t^k{\hat V}^{\prime}(t) = \int_0^ts^{-3} F(V({1}/{s})) ds = \int_0^{t_o}s^{-3} F(V({1}/{s})) ds
+ \int_{t_o}^t s^{-3} F(V({1}/{s})) ds .
\end{align*}
The first integral in the right hand side is certainly finite (by Lemma \ref{morethanexpV} and \eqref{Fin0}). The second integral can be estimated by \eqref{Fcresce} and \eqref{radiallemma}. If $-3+p(k-1)/{2}\neq -1$, we have
\begin{align*}
-t^k{\hat V}^{\prime}(t) \le {\const}(1+ t^{-2+p(k-1)/{2}}) .
\end{align*}
Hence
\begin{align*}
{\hat V}(t) = -\int_t^{+\infty}{\hat V}^{\prime}(s) ds  \le \const(1+ t^{-2+p(k-1)/2}) \, t^{-(k-1)}, 
\end{align*}
and therefore $V(t)   \le \const(1+ t^{2-p(k-1)/2})$.
Now, if $2-p(k-1)/2\ge 0$ we have reached a contradiction. Otherwise we have improved estimate \eqref{radiallemma} to
\begin{align*}
{ V}(t) \le \const t^{2-p(k-1)/2} ,
\end{align*}
with $2-p(k-1)/2>- (k -1)/2 $.
So we can repeat the arguments starting from this better estimate and obtain, after $n$ steps, that $V (t) \le \const(1+ t^{\beta_n})$ for
$\beta_n=2 \sum\limits_{i=0}^n p^i -p^{n+1}\frac{k-1}{2}$.
But $\beta_n= 2 p^{n+1} \left(\frac{1}{p}\sum\limits_{i=0}^n \left(\frac{1}{p}\right)^i -\frac{k-1}{4}\right)$, where $\frac{1}{p}\sum\limits_{i=0}^{\infty} \left(\frac{1}{p}\right)^i -\frac{k-1}{4} = \frac{1}{p-1} -\frac{k-1}{4} >0$ since $p< \frac{k +3}{k-1}$ by assumptions.
Then $\beta_n\to \infty$, which means that after a finite number of steps we obtain that  $V (t) \le \const(1+ t^{\beta_n})$ with $\beta_n\ge0$, i.e.~that $V(0)$ is finite.
\\
If, at some step, we have $-3+p\beta_n= -1$, we can conclude anyway. In fact integrating \eqref{eqkelvin} and using  \eqref{Fcresce} brings to
\begin{align*}
-t^k{\hat V}^{\prime}(t) \le {\const}(1+ \log t) .
\end{align*}
Hence
\begin{align*}
{\hat V}(t) = -\int_t^{+\infty}{\hat V}^{\prime}(s) ds  \le \const(1+ \log t) \, t^{-(k-1)}, 
\end{align*}
and therefore $V(t)   \le \const |\log t|$.
The following iteration gives
\begin{align*}
-t^k{\hat V}^{\prime}(t) \le \const \left(1 + \int_{t_o}^ts^{-3}\log^p s \, ds \right) \le  \const (1 + t^{-1}) ,
\end{align*}
so that $\hat V (t)\le \const (1+t^{-1})\, t^{-(k-1)}$ and eventually $V (t)\le \const$
\end{proof}

We next prove well-posedness of the ODE \eqref{1radtr}, under hypotheses  \eqref{Fin0}, \eqref{Fcresce}--\eqref{F3}.

\begin{proof}[Proof of Proposition \ref{p2.2}]
We first produce a solution to 
\begin{equation}\label{1radtrva}
\left\{\begin{array}{ll}
- V^{\prime\prime}- \dfrac{k(\alpha)}{t} V^{\prime}= F(|V|) , \qquad   & t>0 , \\ 
\lim\limits_{t\to+\infty}V(t)=0, &
\end{array}\right.\end{equation}
 by solving a constrained minimization problem. This part of the proof is next to \cite[Theorem 1]{BL-ARMA}, so we outline it only enlightening the differences.
In the space $E$ assigned by \eqref{E}, we introduce the functionals
\[
{\mathcal T} (V) = \int_0^{+\infty} t^k (V')^2 dt , \qquad
{\mathcal W} (V) = \int_0^{+\infty} t^k G(V) dt , 
\]
where $G$ is  a primitive function for $F(|s|)$, namely
$ G(v)= \int_0^v F(|s|) ds$.
It is needed to minimize the functional $\mathcal T(V)$ on the set $E$, subject to the constraint ${\mathcal W}(V)=1$.
The assumption \eqref{Fnozero}  assures that the set $\{ V\in E \, :\, {\mathcal W}(V)=1 \}$ is not empty (see \cite[Theorem 1, step 1]{BL-ARMA}). Hence there is a minimizing sequence $V_n$, which clearly satisfies 
\begin{equation}\label{V'bound} \int_0^{+\infty}t^k (V_n')^2 dt \le \const. 
\end{equation}
Following \cite[Theorem 1, step 2]{BL-ARMA}, we split the function $F$ into $F(v)=F_1(v)-F_2(v)$ with the properties
\begin{align*}
\lim\limits_{v\to 0}\frac{F_1(v)}{v}=0 , \quad \lim\limits_{v\to+\infty}\frac{F_1(v)}{v^{p_\a}} =0 , \quad F_2(v) \ge m v \; \mbox{ for } v\ge 0 ,
\end{align*}
where $p_\a=\frac{\n+2+2\alpha}{\n-2}>p$ because $\alpha>\frac{(\n-2)p-\n-2}{2}$.
We also denote by $G_i(v)=\int_0^v F_i(|s|)ds$ the respective primitive functions, in such a way that
\[ {\mathcal W}(V) = \int_0^{+\infty}t^k G_1(V) dt - \int_0^{+\infty}t^k G_2(V)dt ,\]
with
\begin{align*}
G_1(V) \le \const(\varepsilon)|V|^{p_{\alpha}+1} + \varepsilon G_2(V) , \quad G_2(V) \ge  \frac{m}{2} |V|^2 .
\end{align*}
Hence along the minimizing sequence we have
\begin{align*}
\int_0^{+\infty}t^k G_2(V_n)dt = & \int_0^{+\infty}t^k G_1(V_n) dt -1 \\
\le & \const(\varepsilon) \int_0^{+ \infty}t^k|V_n|^{p_{\alpha}+1}dt  + \varepsilon\int_0^{+\infty}t^k G_2(V_n)dt -1,
\end{align*}
which gives (taking $\varepsilon=1/2$)
\begin{align}\label{finqui}
\frac{m}{4}\int\limits_0^{+\infty}t^k V_n^2dt \le \frac{1}{2}\int\limits_0^{+\infty}t^k G_2(V_n)dt \le \const \int\limits_0^{+\infty}t^k|V_n|^{p_{\alpha}+1}dt  .
\end{align}
The difference here is that we can not take advantage of a compact immersion to control the term on the right side, but we can make use of some weighted Sobolev embedding in the spirit of \cite{SWW07}, namely
 \begin{equation}\label{Vqbound}
\int_0^{+\infty}t^k |V|^{p_\a+1} dt \le \const \left( \int_0^{+\infty} t^k (V')^2 dt \right)^{\frac{p_{\alpha}+1}{2 }}
\end{equation}
for all $V\in E$. 
The starting point is a decay property of radial functions, which was pointed out in the Radial Lemma by Ni \cite{Ni}. Specifically, for any $V\in E$ and $r>t>0$ we have
\[ V(t) = V(r) - \int_t^r  V'(s) ds , \]
and then
\begin{align*}
|V(t)| \le & |V(r)| + \int_t^r | V'(s)| ds \le |V(r)| + \left(\int_t^r s^{-k}  ds\right)^{1/2} \left(\int_t^r s^k (V'(s))^2 ds\right)^{1/2} \\
\le & |V(r)| + \frac{1}{\sqrt{k-1}}  \left(\int_0^{+\infty} t^k (V')^2 dt\right)^{\frac{1}{2}}\left(t^{1-k}-r^{1-k}\right)^{\frac{1}{2}}.
\end{align*}
Sending $r\to+\infty$, and remembering that $V(r)$ vanishes we obtain
\begin{align}\label{radiallemma}
\left| V(t)\right|\le \frac{1}{\sqrt{k-1}} \left(\int_0^{+\infty} t^k (V')^2 dt\right)^{\frac{1}{2}} t^{-\frac{k-1}{2}}. 
\end{align}
Next we assume that $V$ has compact support in $[0,+\infty)$ (the thesis follows by a standard density argument). Integrating by parts gives
\[
\int_0^{+\infty}t^k |V|^{p_\a+1} dt = - \frac{p_\a+1}{k+1}\int_0^{+\infty}t^{k+1} |V|^{p_\a-1} V \, V' dt.
\]
So
\begin{align*}
\int_0^{+\infty}t^k |V|^{p_\a+1} dt \le & \const \int_0^{+\infty} t^{k+1}|V|^{p_\a} |V'| dt \le \const \left(\int_0^{+\infty} t^k (V')^2 dt\right)^{\frac{1}{2}} \left(\int_0^{+\infty} t^{k+2}|V|^{2 p_\a} dt\right)^{\frac{1}{2}} \\
= & \const \left(\int_0^{+\infty} t^k (V')^2 dt\right)^{\frac{1}{2}}  \left(\int_0^{+\infty} t^2 |V|^{p_\a-1} \, t^{k}|V|^{ p_\a+1} dt\right)^{\frac{1}{2}}.
\end{align*} 
But by \eqref{radiallemma} and the definition of $p_\a$, $k=k(\alpha)$ we have 
\begin{align*}
t^2 |V|^{p_\a-1} \le & \frac{1}{(k-1)^{\frac{p_\a-1}{2}}} \left(\int_0^{+\infty} t^k (V')^2 dt \right)^{\frac{p_\a-1}{2}}
t^{2-\frac{(k-1)(p_\a-1)}{2}} \\ = & \frac{1}{(k-1)^{\frac{p_\a-1}{2}}} \left(\int_0^{+\infty} t^k (V')^2 dt \right)^{\frac{p_\a-1}{2}}
\end{align*}
and then
\begin{align*}
\int_0^{+\infty}t^k |V|^{p_\a+1} dt \le & \const \left(\int_0^{+\infty} t^k (V')^2 dt \right)^{\frac{p_\a+1}{4}}  \left(\int_0^{+\infty} t^{k}|V|^{ p_\a+1} dt\right)^{\frac{1}{2}},
\end{align*} 
which in turns gives \eqref{Vqbound}.
The uniform bounds \eqref{V'bound}, \eqref{finqui} and \eqref{Vqbound} allow to conclude the arguments of \cite[Theorem 1]{BL-ARMA} and get a solution  of \eqref{1radtrva}, by using a Compactness result by Strauss \cite{Strauss}.

The solution $V$ just produced is clearly positive. Actually, it is nontrivial because of the constraint ${\mathcal W}(V)=1$. So it has to be positive by strong maximum principle. Therefore $V$ actually solves \eqref{1radtr}, and vanishes at infinity.
It is left to check that  $\lim\limits_{t\to 0} V'(t)=0$. This follows as in \cite[Lemma 1]{BL-ARMA}, provided that $V(t)$ has finite limit as $t\to 0$. This is true by Lemma \ref{Vbd0},  because $k(\alpha) <\frac{p+3}{p-1}$ when $\alpha>\frac{(\n-2)p-\n-2}{2}$.

Eventually, uniqueness of positive solutions for \eqref{1radtr} has been proved in \cite{KZ91}, under assumptions \eqref{F2}, \eqref{F3}. 
\end{proof}

Last Lemma inherits the decay at infinity of the radial solutions to 
\begin{equation}\label{eq:zh}
\left\{\begin{array}{ll}
-\Delta \zeta +m|x|^{\alpha}\zeta =  |x|^{\alpha}h(e^{-|x|}) , \qquad & \text{in } \R^{\n} ,\\
\zeta(x)\to 0 , & \text{as } |x|\to\infty,
\end{array}\right.
\end{equation}
where $h$ is any continuous, nondecreasing function with $h(0)=0$, $\lim\limits_{s\to 0} h(s)/s=0$.
It has been used in the proof of Lemma \ref{Tdef}, applied to the function $h(s)=s\,g(s)$.

\begin{lemma}\label{piucheesp}
Let $\zeta$ be a positive radial solution to \eqref{eq:zh}, then  $\zeta$ vanishes at infinity faster than $e^{-|x|}$.
\end{lemma}
\begin{proof}
We perform the change of variable \eqref{cov}, so that $Z(t)=\zeta(x)$  vanishes at infinity and satisfies 
\begin{equation}\label{Zh}
 Z^{\prime\prime} + \dfrac{k}{t} Z^{\prime} - m Z = - h(e^{-(t/\beta)^{\beta}}) ,   \quad t>0 
\end{equation}
for  $\beta=2/(2+\alpha)\in(0,1)$ and $k=(2\n-2+\a)/(2+\a)$.
The thesis is equivalent to 
\[
\lim\limits_{t\to+\infty} e^{(t/\beta)^{\beta}}Z(t)  = 0 .
\]
We begin by studying the homogeneous form of \eqref{Zh}, i.e.
\begin{equation}\label{Z0}
 Z^{\prime\prime} + \dfrac{k}{t} Z^{\prime} - m Z = 0 , \qquad t>0 .
\end{equation}
We set $\nu=(k-1)/2$ and introduce a new variable $U$ in such a way that $Z(t)=t^{-\nu} U(\sqrt{m}\, t)$. By computations $U$ solves a modified Bessel equation
\begin{equation}\label{bessel}
 U^{\prime\prime} + \dfrac{1}{s} U^{\prime} -\left(1+ \dfrac{\nu^2}{s^2}\right) U = 0 , \qquad s>0 .
\end{equation}
A system of fundamental solutions for \eqref{bessel} is given by  two modified Bessel functions 
\begin{align*} 
\jcal_{\nu}(s)=& \left(\frac s 2\right)^{\nu}\sum\limits_{n=0}^{\infty}\dfrac{s^{2n}}{4^nn!\Gamma(n+1+\nu)} , \\
\kcal_{\nu}(s)= & \dfrac{\pi\csc(\pi\nu)}{2} \left(\jcal_{-\nu}(s)-\jcal_{\nu}(s)\right) ,
\end{align*}
 if $\nu$ is not integer, or $\kcal_{\nu}(s)=  \lim\limits_{\mu\to\nu} \kcal_{\mu}(s)$,
if $\nu$ is an integer.
Here $\Gamma$ stands for the usual $\Gamma$-function. Let us recall some useful properties of the modified Bessel functions (see, for instance \cite{book:Korenev}):
\[ \jcal_{\nu}(s) \kcal_{\nu}'(s) -\jcal_{\nu}'(s) \kcal_{\nu}(s) = - 1/s,\]
\begin{align*}
\jcal_{\nu}(s) = &  \dfrac{1}{\Gamma(\nu+1)}\left(\dfrac{s}{2}\right)^{\nu} (1+O(s^2)) ,   \quad \text{ as } s\to 0, 
\\ \jcal_{\nu}(s) = & \dfrac{e^{s}}{\sqrt{2\pi s}} \left(1+0\left(\dfrac{1}{s}\right)\right)  , \quad \text{ as } s\to +\infty , 
\\ \kcal_{\nu}(s) = & \sqrt{\dfrac{\pi}{2}} \dfrac{e^{-s}}{\sqrt{s}} \left(1+0\left(\dfrac{1}{s}\right)\right)  ,  \quad \text{ as } s\to +\infty .
\end{align*}
Now, any solution of the homogeneous ODE  is a linear combination of $t^{-\nu}\jcal_{\nu}(\sqrt{m} \,t)$ and $t^{-\nu}\kcal_{\nu}(\sqrt{m} \,t)$, and solutions to the non-homogeneous ODE in \eqref{Zh} can be produced by the method of the variation of constants
\[Z(t) = a(t)\, t^{-\nu} \jcal_{\nu}(\sqrt{m}\, t) + b(t)\, t^{-\nu}\kcal_{\nu}(\sqrt{m}\, t),\]
\[
a'(t)= h(e^{-(t/\beta)^{\beta}}) \,t^{-\nu}\kcal_{\nu}(\sqrt{m}\, t) / W(t) , \quad b'(t)= - h(e^{-(t/\beta)^{\beta}}) \,t^{-\nu}\jcal_{\nu}(\sqrt{m}\, t) / W(t),
\]
where $W(t)$ is the  Wronskian determinant:
\begin{align*}
 W(t)= & t^{-\nu}\jcal_{\nu}(\sqrt{m} \,t)\left(t^{-\nu} \kcal_{\nu}(\sqrt{m}\, t)\right)' - \left(t^{-\nu}\jcal_{\nu}(\sqrt{m} \,t)\right)'t^{-\nu} \kcal_{\nu}(\sqrt{m} \,t)
\\ = & \sqrt{m}\, t^{-2\nu}\left(\jcal_{\nu} \kcal_{\nu}' -\jcal_{\nu}' \kcal_{\nu}\right) = -t^{-2\nu-1} .
\end{align*}
Hence
\[
a'(t)=  - h(e^{-(t/\beta)^{\beta}}) \,t^{\nu+1}\kcal_{\nu}(\sqrt{m}\, t), \quad b'(t)= h(e^{-(t/\beta)^{\beta}}) \,t^{\nu+1}\jcal_{\nu}(\sqrt{m}\, t) .
\]
By the recalled properties of the Bessel functions $a'$ is integrable at infinity (because also $h(e^{-(t/\beta)^{\beta}})$ vanishes), while $b'$ is integrable at $0$.
We thus write
\begin{align*}
Z(t) = & \left(a_o+\int_t^{+\infty}h(e^{-(s/\beta)^{\beta}}) \,s^{\nu+1}\kcal_{\nu}(\sqrt{m}\, s) ds \right) t^{-\nu} \jcal_{\nu}(\sqrt{m}\, t) \\
& +\left( b_o+\int_0^th(e^{-(s/\beta)^{\beta}}) \,s^{\nu+1}\jcal_{\nu}(\sqrt{m}\, s)  ds\right)  t^{-\nu}\kcal_{\nu}(\sqrt{m}\, t).
\end{align*}
We will show later on  that 
\begin{align}
\label{+} 
\lim\limits_{t\to+\infty} e^{(t/\beta)^{\beta}}t^{-\nu} \jcal_{\nu}(\sqrt{m}\, t)\int_t^{+\infty}h(e^{-(s/\beta)^{\beta}}) \,s^{\nu+1}\kcal_{\nu}(\sqrt{m}\, s) ds = & 0 \\
\label{-} 
\lim\limits_{t\to+\infty} e^{(t/\beta)^{\beta}}t^{-\nu} \kcal_{\nu}(\sqrt{m}\, t)\int_0^t h(e^{-(s/\beta)^{\beta}}) \,s^{\nu+1}\jcal_{\nu}(\sqrt{m}\, s) ds = & 0 .
\end{align}
Hence the condition $\lim\limits_{|x|\to+\infty}\zeta(x)=\lim\limits_{t\to+\infty} Z(t) = 0$ implies that $a_o=0$ and the proof is completed.

To check \eqref{+}, we apply De L'Hopital's theorem and get
\begin{align*}
\lim\limits_{t\to+\infty} e^{(t/\beta)^{\beta}}t^{-\nu} \jcal_{\nu}(\sqrt{m}\, t)\int_t^{+\infty}h(e^{-(s/\beta)^{\beta}}) \,s^{\nu+1}\kcal_{\nu}(\sqrt{m}\, s) ds \\ = 
\lim\limits_{t\to+\infty} \dfrac{\int_t^{+\infty}h(e^{-(s/\beta)^{\beta}}) \,s^{\nu+1}\kcal_{\nu}(\sqrt{m}\, s) ds}{e^{-(t/\beta)^{\beta}}t^{\nu}\left( \jcal_{\nu}(\sqrt{m}\, t)\right)^{-1}}\\ = 
-\lim\limits_{t\to+\infty} \dfrac{h(e^{-(t/\beta)^{\beta}}) \,t^{\nu+1}\kcal_{\nu}(\sqrt{m}\, t) }{
\dfrac{e^{-(t/\beta)^{\beta}}}{\left( \jcal_{\nu}(\sqrt{m}\, t)\right)^2}\left((-{\beta^{1-\beta}}{t^{\nu+\beta-1}} + \nu t^{\nu-1}) \jcal_{\nu}(\sqrt{m}\, t) - \sqrt{m}\, t^{\nu} \jcal_{\nu}'(\sqrt{m}\, t)\right)} \\=
- \lim\limits_{t\to+\infty} h(e^{-(t/\beta)^{\beta}}) e^{(t/\beta)^{\beta}} \dfrac{t \kcal_{\nu}(\sqrt{m}\, t) \jcal_{\nu}(\sqrt{m}\, t)}{-\left(\dfrac{\beta}{t}\right)^{1-\beta} + \dfrac{\nu}{t}  - \sqrt{m} \dfrac{\jcal_{\nu}'(\sqrt{m}\, t)}{\jcal_{\nu}(\sqrt{m}\, t)}}.
\end{align*}
Here, the term $h(e^{-(t/\beta)^{\beta}}) e^{(t/\beta)^{\beta}}$ vanishes by assumption and the numerator of the fraction is bounded. Concerning the denominator, remembering that \[\jcal_{\nu}'(s)= \dfrac{\nu}{s} \jcal_{\nu}(s) + \jcal_{\nu+1}(s)\]  we have 
\begin{align*}
-\left(\dfrac{\beta}{t}\right)^{1-\beta} + \dfrac{\nu}{t}  - \sqrt{m} \dfrac{\jcal_{\nu}'(\sqrt{m}\, t)}{\jcal_{\nu}(\sqrt{m}\, t)}
=-\left(\dfrac{\beta}{t}\right)^{1-\beta} -  \sqrt{m} \dfrac{\jcal_{\nu+1}(\sqrt{m}\, t)}{\jcal_{\nu}(\sqrt{m}\, t)},
\end{align*}
where the first term vanishes and the second one has finite and positive limit because   $\jcal_{\nu+1}$ and $\jcal_{\nu}$ have the same asymptotic behavior at infinity.

Concerning \eqref{-}, we first notice that $\lim\limits_{t\to+\infty}\int_0^t h(e^{-(s/\beta)^{\beta}}) \,s^{\nu+1}\jcal_{\nu}(\sqrt{m}\, s) ds$ exists because the integrand is nonnegative. If the limit is finite, there is nothing to prove. Otherwise, we can apply De L'Hopital's theorem and computations similar to the previous ones give the thesis. Actually
\begin{align*}
\lim\limits_{t\to+\infty} e^{(t/\beta)^{\beta}}t^{-\nu} \kcal_{\nu}(\sqrt{m}\, t)\int_0^t h(e^{-(s/\beta)^{\beta}}) \,s^{\nu+1}\jcal_{\nu}(\sqrt{m}\, s) ds \\ = 
\lim\limits_{t\to+\infty}\dfrac{\int_0^t h(e^{-(s/\beta)^{\beta}}) \,s^{\nu+1}\jcal_{\nu}(\sqrt{m}\, s) ds}{e^{-(t/\beta)^{\beta}}t^{\nu} \left(\kcal_{\nu}(\sqrt{m}\, t)\right)^{-1}} \\ =
\lim\limits_{t\to+\infty}\dfrac{h(e^{-(t/\beta)^{\beta}}) \,t^{\nu+1}\jcal_{\nu}(\sqrt{m}\, t)}{\frac{e^{-(t/\beta)^{\beta}}}{\left(\kcal_{\nu}(\sqrt{m}\, t)\right)^2} \left((-\beta^{1-\beta}t^{\nu+\beta-1} + \nu t^{\nu-1}) \kcal_{\nu}(\sqrt{m}\, t) - \sqrt{m}\, t^{\nu} \kcal_{\nu}'(\sqrt{m}\, t)\right)} \\ =
\lim\limits_{t\to+\infty} h(e^{-(t/\beta)^{\beta}}) e^{(t/\beta)^{\beta}} \dfrac{t\, \jcal_{\nu}(\sqrt{m}\, t) \kcal_{\nu}(\sqrt{m}\, t)}{-\left(\dfrac{\beta}{t}\right)^{1-\beta} + \dfrac{\nu}{t}- \sqrt{m}\dfrac{\kcal_{\nu}'(\sqrt{m}\, t)}{\kcal_{\nu}(\sqrt{m}\, t)} } .
\end{align*}
Here, the term $h(e^{-(t/\beta)^{\beta}}) e^{(t/\beta)^{\beta}}$ vanishes by assumption and the numerator of the fraction is bounded. Concerning the denominator, remembering that \[\kcal_{\nu}'(s)= \dfrac{\nu}{s} \kcal_{\nu}(s) - \kcal_{\nu+1}(s)\]  we have 
\begin{align*}
-\left(\dfrac{\beta}{t}\right)^{1-\beta} + \dfrac{\nu}{t}  - \sqrt{m} \dfrac{\kcal_{\nu}'(\sqrt{m}\, t)}{\kcal_{\nu}(\sqrt{m}\, t)}
=-\left(\dfrac{\beta}{t}\right)^{1-\beta} +  \sqrt{m} \dfrac{\kcal_{\nu+1}(\sqrt{m}\, t)}{\kcal_{\nu}(\sqrt{m}\, t)},
\end{align*}
where the first term vanishes and the second one has finite and positive limit because   $\kcal_{\nu+1}$ and $\kcal_{\nu}$ have the same asymptotic behavior at infinity.
\end{proof}

\end{document}